\documentclass[11pt]{article}
\usepackage{amsmath}
\usepackage{enumerate}
\usepackage{graphicx}
\usepackage{multicol}
\usepackage{amsfonts} 
\usepackage{amsthm}

\usepackage{float}
\floatstyle{boxed} 
\restylefloat{figure}

\newcommand\tab[1][1cm]{\hspace*{#1}}

\newtheorem{theorem}{Theorem}

\begin{document}

\begin{center}
\Large
\textbf{$\downarrow$-posets}

by Lawrence Valby

September, 2016

lvalby@iu.edu, lawrence.valby@gmail.com
\end{center}

\section{Abstract}

We investigate a certain class of posets arising from semilattice actions. Let $S$ be a semilattice with identity. Let $S$ act on a set $C$. For $c,d\in C$ put $c\leq d$ iff there is some $s\in S$ with $ds=c$. Then $(C,\leq)$ is a poset. Let's call the posets that arise in this way $\downarrow$-posets. We give a reasonable second order characterization of $\downarrow$-posets and show that there is no first order characterization.

\section{Introduction}

The original motivation for this paper came from Seth Yalcin and Daniel Rothchild and their work on the semantics and pragmatics of natural language conversations \cite{Yalcin}, but the results described here are of a purely mathematical character, and may be of interest to people who study posets, combinatorial topologists ($\downarrow$-functions below are related to deformation retractions), and logicians or computer scientists (e.g.\ is recognizing a $\downarrow$-poset $\mathrm{P}$ or $\mathrm{NP}$-complete?). In addition to thanking Seth Yalcin for bringing the problem to my attention, I would also like to thank George Bergman for valuable input.

A \textbf{semilattice with identity} is a set $S$ together with a binary operation $S\times S\to S$ (written concatenatively) and an element $1\in S$ satisfying the following universal equations:
\begin{enumerate}
\item $ss=s$ \tab\tab (idempotent)
\item $st=ts$ \tab\tab (commutative)
\item $(st)u=s(tu)$ \tab (associative)
\item $1s=s$ \tab\tab (identity)
\end{enumerate}
The quintessential example of a semilattice with identity is a subset $S$ of the powerset of some set $W$, i.e.\ $S\subseteq \mathcal{P}(W)$, and the operation is given by intersection, i.e.\ $st:=s\cap t$, and the identity element is the whole set $W$, i.e.\ $1:=W$. Indeed, it is straightforward to prove that every semilattice with identity is isomorphic to such a semilattice with identity.

Let $S$ be a semilattice with identity. We describe what it means for $S$ to \textbf{act on} a set $C$. It means that there is a function $C\times S\to C$ (the \textbf{action}, written concatenatively) that satisfies the following universal equations.
\begin{enumerate}
\item $c(st)=(cs)t$\tab (and so we may write simply $cst$)
\item $c1=c$
\end{enumerate}
 Notice that we generally use lowercase letters like $c,d,e$ to denote elements of the set $C$, and lowercase letters like $s,t,u$ to denote elements of the semilattice $S$.

Let $S$ act on $C$. We define a binary relation $\leq$ on $C$ as follows. We put $c\leq d$ iff there is some $s\in S$ such that $ds=c$. Any such $\leq$ is in fact a poset; as an example we verify antisymmetry. Let $cs=d$ and $dt=c$. Then $ct=dtt=dt=c$ and so $c=dt=cst=cts=cs=d$. We call the posets that arise in this way (or isomorphic to such a poset) \textbf{$\downarrow$-posets}. Intuitively, we think of elements of $C$ as states of some system, the elements of $S$ as specific acts that may be performed that change the state of the system, and then $c\leq d$ means that it is possible to transition from state $d$ to state $c$. 

The quintessential example of a $\downarrow$-poset is one that arises as follows. The set $C$ is a subset of the powerset of some set $W$, i.e.\ $C\subseteq \mathcal{P}(W)$, the semilattice $S$ is a sub-semilattice of $(\mathcal{P}(W),\cap,W)$, the action $C\times S\to C$ is given by intersection, and the induced order on $C$ matches $\subseteq$. In fact, every $\downarrow$-poset is isomorphic to one that arises in this way. 

For example, let $W=\{1,2,3,4\}$, $C=\{c,d,e,f,g\}$, and $S=\{s,t,u,\mathrm{id}\}$ where
\begin{multicols}{2}
\begin{itemize}
\item $c=\emptyset$
\item $d=\{1\}$
\item $e=\{2\}$
\item $f=\{1,2,3\}$
\item $g=\{1,2,4\}$
\end{itemize}
\columnbreak
\begin{itemize}
\item $s=\emptyset$
\item $t=\{1\}$
\item $u=\{2\}$
\item $\mathrm{id}=\{1,2,3,4\}$
\end{itemize}
\end{multicols}
\noindent Figure~\ref{not_semilattice_yes} is a picture of the resulting $\downarrow$-poset (up the page is up in the poset).

\begin{figure}
\begin{center}
\includegraphics[scale=0.2, angle=360]{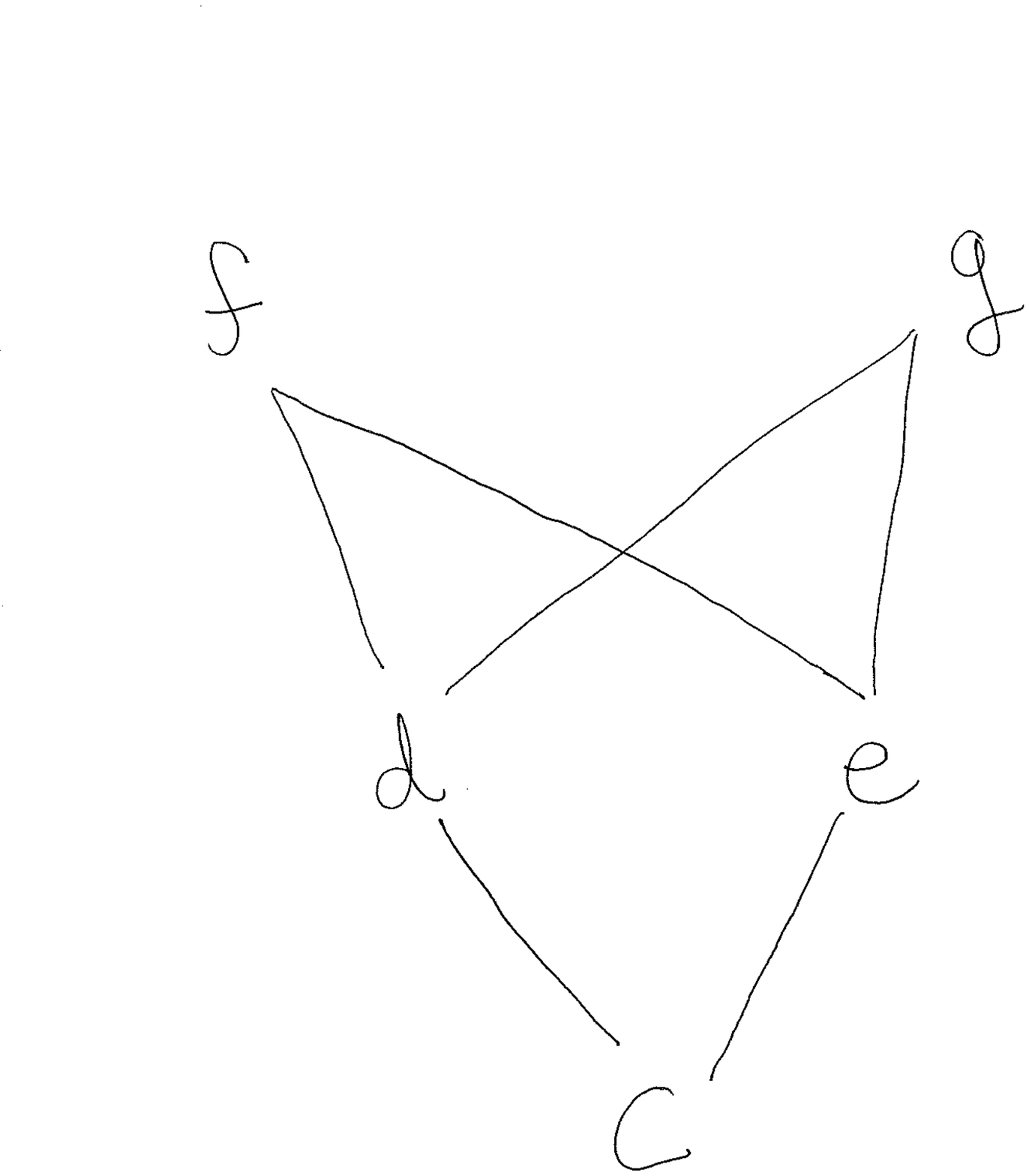}
\end{center}
\caption{Example of a $\downarrow$-poset}
\label{not_semilattice_yes}
\end{figure}

This paper is concerned with the question of characterizing which posets are $\downarrow$-posets. Every semilattice is a poset (via $x\leq y$ iff $xy=x$), and every semilattice is a $\downarrow$-poset, but \textbf{not every $\downarrow$-poset is a semilattice}. Figure~\ref{not_semilattice_yes} furnishes an example of a $\downarrow$-poset that is not a semilattice. Furthermore, not all posets are $\downarrow$-posets. For example, consider the poset in Figure~\ref{not_semilattice_not}. If it were a $\downarrow$-poset, then there would be some action $C\times S\to C$ that gave rise to this order. In particular, we would have $es=c$ and $et=d$ for some $s,t\in S$. But then $c=est=ets=d$, a contradiction.

\begin{figure}
\begin{center}
\includegraphics[scale=0.2, angle=90]{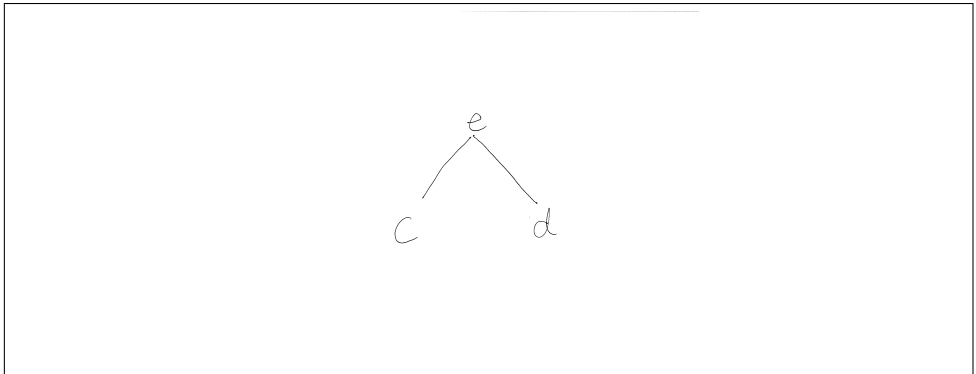}
\end{center}
\caption{A poset that is not a $\downarrow$-poset}
\label{not_semilattice_not}
\end{figure}

To give the reader a sense of the problem, consider the poset in Figure~\ref{not_down_poset} and see if you can tell whether it is a $\downarrow$-poset or not. We will answer this question in the next section, using our second order characterization of $\downarrow$-posets.

\begin{figure}
\begin{center}
\includegraphics[scale=0.2, angle=360]{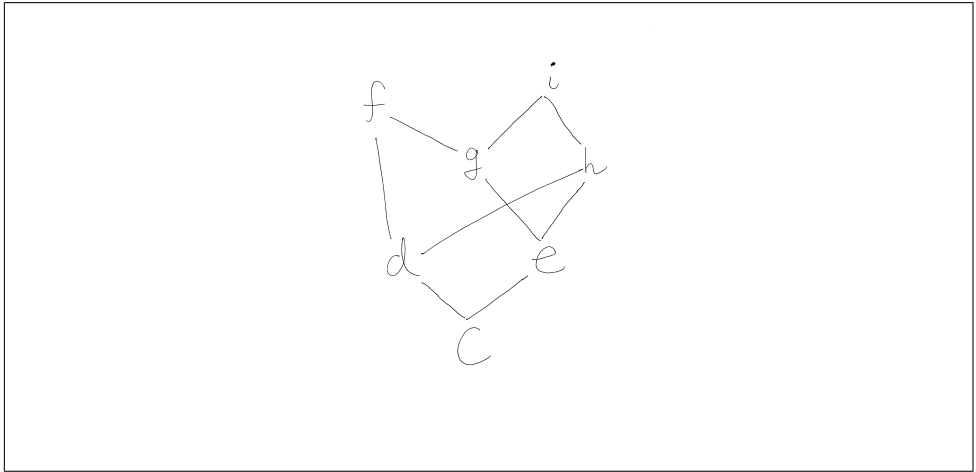}
\end{center}
\caption{A $\downarrow$-poset, or not?}
\label{not_down_poset}
\end{figure}

\section{Second Order Characterization}

Let $(C,\leq)$ be a poset. We call a function $f\colon C\to C$ a \textbf{$\downarrow$-function} when
\begin{enumerate}
\item $f(c)\leq c$ \tab\tab\tab\tab (decreasing)
\item $c\leq d$ implies $f(c)\leq f(d)$ \tab (monotone)
\item $c\leq f(d)$ implies $f(c)=c$ \tab (``below-image-fixing")
\end{enumerate} 
An example of a $\downarrow$-function is, in the context of a $\downarrow$-poset and a fixed $s\in S$, the function $c\mapsto cs$.  In detail, this is decreasing because $cs\leq c$ by the definition of $\leq$, it is monotone because $dt=c$ implies $dst=dts=cs$, and it is below-image-fixing because $dst=c$ implies $cs=dsts=dst=c$. 

By way of introduction to $\downarrow$-functions, we discuss how the definition of $\downarrow$-function is similar to that of interior operator in topology or box operator in modal logic, yeat also different. Let $(C,\leq,\wedge,1)$ be a poset with finite meets. A function $\Box\colon C\to C$ is an interior operator when for $x,y\in C$:
\begin{enumerate}
\item $\Box(1)=1$
\item $\Box(x\wedge y)=\Box x\wedge\Box y$
\item $\Box x\leq x$
\item $\Box\Box x=\Box x$
\end{enumerate}
Both $\downarrow$-functions and interior operators are decreasing, and both are monotone (in the case of $\Box$ this follows from preservation of $\wedge$). Binary meets $\wedge$ might not exist in the case of an arbitrary poset, but all binary meets that do exist are preserved by $\downarrow$-functions, so they have this in common as well. Furthermore, because $f(c)\leq f(c)$, we get $ff(c)=f(c)$ as a special case of below-image-fixing. However, on the other hand, a $\downarrow$-function might not preserve 1, even if 1 exists, and so this is a difference. Another difference is that $\Box$ might not satisfy below-image-fixing. E.g., consider the usual topology on $\mathbb{R}$ and note that $[0,1]\subseteq\Box\mathbb{R}=\mathbb{R}$, but $\Box[0,1]=(0,1)\neq [0,1]$. 

\begin{theorem}\label{second order}
A poset $(C,\leq)$ is a $\downarrow$-poset iff whenever $c\leq d$ there is a $\downarrow$-function $f\colon C\to C$ such that $f(d)=c$ (we say ``there are enough $\downarrow$-functions").
\end{theorem}

\begin{proof}
Suppose first that $(C,\leq)$ is a $\downarrow$-poset. Let $c\leq d$. Then there is some $s$ such that $ds=c$, and the function $e\mapsto es$ is a $\downarrow$-function.

Now conversely suppose that $(C,\leq)$ is a poset with enough $\downarrow$-functions. Let $S$ be the set consisting of the $\downarrow$-functions of $C$. We claim that $S$ is a semilattice with identity where the product is the usual composition of functions and 1 is the identity function. This is easy to verify; we include as an example a proof that the composition of $\downarrow$-functions is commutative. Let $f$ and $g$ be $\downarrow$-functions. We wish to show $f\circ g=g\circ f$. It suffices to show that $fgc \leq gfc$ for each $c\in C$.  Because $g$ is decreasing, $gc \leq c$. Then, because $f$ is monotone, $fgc \leq fc$. Because $g$ is monotone, we in turn get $gfgc \leq gfc$. We claim $gfgc = fgc$, which will establish $fgc \leq gfc$. Because$f$ is decreasing, $fgc \leq gc$. So, because $g$ is below-image-fixing, $gfgc = fgc$.

This semilattice with identity $S$ acts on $C$ in the obvious way. Finally, the $\downarrow$-poset $(C,\leq')$ this action induces is the same as the original poset $(C,\leq)$. As a relation $\leq'$ is a subset of $\leq$ because $\downarrow$-function are decreasing, and $\leq$ is a subset of $\leq'$ because of the assumption that there are enough $\downarrow$-functions.
\end{proof}

To illustrate this theorem, let's return to the poset of Figure~3 and see why it's not a $\downarrow$-poset. If it were, then there would be a $\downarrow$-function $\alpha$ such that $\alpha(i)=h$. Because $d,e\leq h$ and $h$ is in the image of $\alpha$, and $\alpha$ is below-image-fixing, we get that $\alpha(d)=d$ and $\alpha(e)=e$. As $d,e\leq f$, by monotonicity of $\alpha$ we get $d,e\leq\alpha(f)$. Thus, $\alpha(f)=f$ because $\alpha$ is decreasing. This implies that $\alpha(g)=g$ (by below-image-fixing), and so because $g\leq i$ we get by monotonicity $g=\alpha(g)\leq\alpha(i)=h$, which is a contradiction. 

I note that the $\downarrow$-functions of any poset form a semilattice with identity as in the proof of Theorem~\ref{second order}, whether or not the original poset was a $\downarrow$-poset. In this way, every poset $(C,\leq)$ has a canonical $\downarrow$-poset $(C,\leq')$ inside of it, in the sense that $\leq'$ is a subset of $\leq$. The nature of these canonical $\downarrow$-posets and their relationship to the original posets is a matter for further study. Some basic observations about the situation are as follows. First, $\leq'$ is not always a maximal subset of $\leq$ that makes $C$ into $\downarrow$-poset. Second, although every $\leq$-$\downarrow$-function is a $\leq'$-$\downarrow$-function, the reverse is not always true.

\section{No First Order Characterization}

While there is a succinct second order characterization of $\downarrow$-posets, there is no first order characterization. The class of $\downarrow$-posets is closed under ultraproducts (as it is definable by an existential second order sentence, implicit in Theorem~\ref{second order}), but its complement is not closed under ultrapowers.

\begin{theorem}
There is no first order axiomatization of $\downarrow$-posets (in the signature only containing $\leq$).
\end{theorem}
\begin{proof}[Proof Sketch]
One can give an example of a poset that is not a $\downarrow$-poset, but has an ultrapower that is. The poset pictured in Figure~\ref{not_first_order} is such a poset. The basic idea of the example is that if $f$ were a $\downarrow$-function with $f(b)=a$, then $f(n)=n$ for each $n\in\mathbb{N}$ (using $z_n$ and the $u$'s and $x$'s), and as $f(w)$ must be 0, there is nothing that $f(y)$ can be. On the other hand, a suitable ultrapower of this poset will have an ``infinite natural number" below $y$ that can work as $f(y)$. Of course, one has to carefully check that there are no other obstructions to the ultrapower being a $\downarrow$-poset.
\end{proof}

\begin{figure}
\begin{center}
\includegraphics[scale=0.3,angle=90]{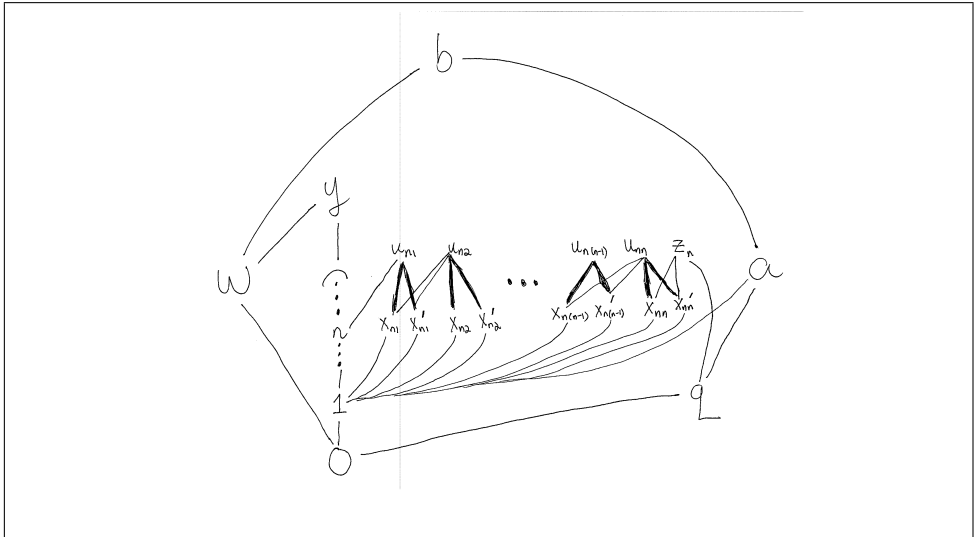}
\end{center}
\caption{Not a $\downarrow$-poset, but suitable ultrapower is}
\label{not_first_order}
\end{figure}

\newpage

\section{Further Questions}

Some questions for further study include the following. 

\begin{enumerate}
\item Is there a first order sentence that is satisfied by a \textbf{finite} poset iff it is a $\downarrow$-poset?
\item Relatedly, is there a polynomial time algorithm that decides whether a finite poset is a $\downarrow$-poset? (The second order characterization gives an $\mathrm{NP}$ algorithm.)
\item Investigate the connection between $\downarrow$-functions and the deformation retractions of combinatorial topology.
\item Every poset has a canonical $\downarrow$-poset inside it. Investigate this.
\item Does this result have any bearing on models for the semantics and pragmatics of natural languages?
\end{enumerate}

\end{document}